\documentclass[a4paper,11pt,reqno]{amsart}

\usepackage[utf8]{inputenc}
\usepackage[T1]{fontenc}
\usepackage[english]{babel}
\usepackage{lmodern}
\usepackage{amsmath,amssymb,amsthm,amscd}
\usepackage{a4wide,mathrsfs,graphicx,enumerate,textcomp,xspace,calc}
\usepackage{xcolor}

\definecolor{darkgreen}{rgb}{0,0.5,0}
\definecolor{darkred}{rgb}{0.7,0,0}
\usepackage[colorlinks, 
citecolor=darkgreen, linkcolor=darkred]{hyperref}

\newtheorem{defn}{Definition}[section]
\newtheorem{lemma}[defn]{Lemma}
\newtheorem{prop}[defn]{Proposition}
\newtheorem{thm}[defn]{Theorem}
\newtheorem{cor}[defn]{Corollary}

\newtheorem{rk}[defn]{Remark}

\newtheorem{mtheorem}{Theorem}

\newtheorem{mcor}[mtheorem]{Corollary}

\def\RR{\mathbb{R}}
\def\Ric{\mathop{\rm Ric}\nolimits}
\def\lie{\mathcal{L}}

\def\Rm{\mathop{\rm Rm}\nolimits}
\def\tr{\mathop{\rm tr}\nolimits}

\def\div{\mathop{\rm div}\nolimits}

\def\Sym{\mathop{\rm Sym}\nolimits}

\newsavebox\CBox
\newcommand\hcancel[2][0.5pt]{%
  \ifmmode\sbox\CBox{$#2$}\else\sbox\CBox{#2}\fi%
  \makebox[0pt][l]{\usebox\CBox}%
  \rule[0.5\ht\CBox-#1/2]{\wd\CBox}{#1}}

\numberwithin{equation}{section}

\title{Linear stability of Steady and Expanding K\"ahler-Ricci Solitons}
\date{}

\author{Lucas Lavoyer}
\address[Lucas Lavoyer]{Mathematisches Institut, Universit\"at M\"unster, 48149 M\"unster, Germany}
\email{lucas.lavoyer@uni-muenster.de}

\author{Adam Thompson}
\address[Adam Thompson]{Mathematisches Institut, Universit\"at M\"unster, 48149 M\"unster, Germany}
\email{a.thompson@uni-muenster.de}

\begin{document}
\raggedbottom

\begin{abstract}
We prove linear stability of all steady and expanding gradient K\"ahler-Ricci solitons. In the expanding case, we prove strict linear stability under very general assumptions. In particular, every asymptotically conical expanding gradient K\"ahler-Ricci soliton is strictly linearly stable. 
\end{abstract}

\maketitle

\section{Introduction}

A \textit{Ricci soliton} is a triple $(M^n,g,X)$ where $(M^n,g)$ is a Riemannian manifold and $X$ is a vector field satisfying the equation 
\begin{equation*}
\Ric(g)-\frac{1}{2}\lie_X(g)+\frac{\varepsilon}{2} g=0.
\end{equation*}
 We call $X$ the \textit{soliton vector field}. Moreover, if $X=\nabla f$ for some real-valued smooth function $f$ on $M$ called the potential function then $(M^n,g,\nabla f)$ is said to be a \textit{gradient} soliton. In this paper, we focus on steady ($\varepsilon =0$) and expanding ($\varepsilon = 1$) gradient Ricci solitons, whose equation reduces to
\begin{equation}
\Ric(g)-\nabla^{2}f+\frac{\varepsilon}{2} g=0.\label{egs-sol-eqn}
\end{equation}

If $(M^n,g,\nabla f)$ denotes a gradient Ricci soliton, the \textit{weighted Laplacian} of a tensor $T$ on $M$ denoted by $\Delta_fT$ is defined by:
\begin{equation*}
\Delta_fT:=\Delta T+\nabla_{\nabla f}T,
\end{equation*}
where $\Delta $ denotes the rough Laplacian associated to the Riemannian metric $g$. The \textit{weighted Lichnerowicz operator} associated to a gradient Ricci soliton is given by 
\begin{equation*}
L_f = - \Delta_f - 2\Rm \ast.
\end{equation*}

Let $d\mu_f := e^f d\mu_g,$ where $d\mu_g$ is the standard (Riemannian) volume form associated to the metric $g.$ We define the weighted $L^2$ space by $L_f^2 := L^2(d\mu_f).$ A priori, $L_f h$ is well defined for every $h \in C_c^{\infty}(M, \Sym^2 T^{\ast}M).$ The definition below is as in \cite[Definition 1.1]{DeruelleStability}.

\begin{defn}\label{def; linear stability}
A gradient Ricci soliton $(M^n,g,\nabla f)$ is \textit{linearly stable} if, for every smooth, compactly supported, symmetric 2-tensor $h,$ it satisfies
\begin{equation*}
Q_f(h):= \int_M \langle L_f h,h\rangle d\mu_f = \int_M \left( |\nabla h|^2 - \langle 2\Rm\ast h, h\rangle \right)d\mu_f \geq 0.
\end{equation*}
If there exists a constant $\lambda >0$ such that $Q_f(h)\geq \lambda \int_M |h|^2 d\mu_f,$ then $(M,g,\nabla f)$ is \textit{strictly stable}.
\end{defn}

\begin{mtheorem}\label{thm; linear stability}
Every non-compact, complete steady or expanding gradient K\"ahler-Ricci soliton is linearly stable.
\end{mtheorem}

In recent years, there has been a great effort to develop a theory of Ricci flow through singularities in higher dimensions \cite{BamlerSurvey}. In dimension 4, there is some expectation that the remarkable results obtained in dimension 3 \cite{Perelman2003, BamKlein} could be generalised. In dimension 5 or higher, however, the uniqueness of Ricci flow through singularities is not expected in general \cite{AngenentKnopf}; we should look, therefore, for a \textit{generic Ricci flow} through singularities. In this context, it is expected that a generic Ricci flow developing a singularity only encounters linearly stable shrinking and steady solitons as its singularity models, and these models can then be treated via well-behaved surgery procedures. Such solitons should also be the canonical singularity models: the idea being that unstable singularity models can be avoided via a perturbation of the initial metric, and a generic flow should then only encounter stable ones. We note that for compact shrinking solitons, the stability question has been heavily studied; see \cite{CaoZhuCompact, CaoHamiltonIlmanen,CaoZhuSecondVariation} and the references therein. As a source of inspiration for the heuristics above, we note that this has been achieved in the case of mean curvature flow on embedded surfaces in $\mathbb R^3,$ see \cite{CM12,genMCF, genMCF2, BK23} and the references therein.

As part of the study of Ricci flow through singularities in higher dimensions, isolated conical singularities are expected to often arise at first singular times \cite{BamlerSurvey}, and it has been proposed \cite{Gianniotis-Schulze} that expanding gradient Ricci solitons could be used to resolve such singularities. With this goal in mind, there has been a lot of recent interest in asymptotically conical (AC) expanding gradient Ricci solitons \cite{DerGAFA, DeruelleStability, Der-Lamm, CLP24, BamlerChen}. In the K\"ahler setting, the work of Conlon, Deruelle and Sun \cite{ConlonDeruelleJDG, ConlonDeruelleSun} shows existence and uniqueness of expanding gradient K\"ahler-Ricci solitons out of any K\"ahler cone admitting a smooth canonical model, and \cite{ConlonDeruelleSun} shows that this condition on the K\"ahler cone is, in fact, necessary and sufficient for it to admit a smoothing via an expanding K\"ahler-Ricci soliton. This contrasts with the Riemannian case, where the best known results usually require some positive curvature assumption on the cone.

The stronger results in the K\"ahler case raise the question of general uniqueness of these solutions without restricting ourselves to the K\"ahler-Ricci flow. In this direction, Chen \cite{LongtengUniqueness, LongtengStability} shows uniqueness and stability of the K\"ahler-Ricci flow associated to an AC expanding gradient K\"ahler-Ricci soliton within the expander's K\"ahler class, assuming certain conditions hold; building on these ideas, \cite{CHL26a} shows that a solution flowing out of a conical singularity, in many cases, satisfies a type I curvature bound and is modelled on the associated K\"ahler-Ricci expander. See also \cite{BamlerChenConlon}, where the authors give an example of a K\"ahler cone that does not admit a smoothing by a K\"ahler expander by \cite{ConlonDeruelleSun}, but it is smoothed out by a non-K\"ahler expander.

Inspired by this, we would like to investigate the uniqueness of steady and expanding K\"ahler-Ricci solitons in greater generality, i.e. not just as the unique solitons, or the unique K\"ahler-Ricci flows, but as the unique flows out of such cones. Studying stability of these solutions is then a natural first step. In real dimension 4, the work of Naff and Ozuch \cite{NaffOzuch} shows that steady and expanding gradient K\"ahler-Ricci solitons are stable in the sense of Definition \ref{def; linear stability}. Theorem \ref{thm; linear stability}, then, generalises their result to every dimension. Strict stability often yields dynamical consequences and should imply local uniqueness, so we also study strict stability in the expanding case. We consider the following set of possible conditions, and explain how they naturally fit into the programme above.

 \medskip
\noindent\textbf{Condition $(\mathscr{E})$.}
We say an expanding gradient K\"ahler--Ricci soliton $(M,g,\nabla f)$
satisfies $(\mathscr{E})$ if $\sup_M |\Rm(g)|< \infty$ and
at least one of the following holds:
\begin{enumerate}[(i)]
    \item\label{condition; f is proper}
    $f$ is proper;
    \item\label{condition; R to 0}
    $R(x)\to 0$ as $x\to\infty$;
    \item\label{condition; R nonneg}
    $R\geq 0$ on $M$.
\end{enumerate}
For each of these, we show strict stability for the expanding soliton. We note that assumption \eqref{condition; R to 0} holds, for instance, whenever the expanding soliton being considered is asymptotically conical to a smooth cone. Assumption \eqref{condition; R nonneg} is also natural in the context of flowing out of singularities: the expected small-scale behaviour of a Ricci flow through singularities is described by a shrinker-cone-expander transition \cite{JianSong}, with both the shrinker and the expander being asymptotic to the cone (see \cite{CHL26b} for a proof of this behaviour in complex dimension 2 and in higher dimensional K\"ahler-Ricci flow under the Calabi ansatz). In this case, since the shrinker has nonnegative scalar curvature \cite{BinglongChen}, it is expected that the cone and expander should also have this property.

\begin{mtheorem}\label{thm; strict stability}
Let $(M^{2m},g,\nabla f)$ be a non-compact, complete expanding gradient K\"ahler-Ricci soliton satisfying $(\mathscr{E})$. Then it is strictly linearly stable. Furthermore, if we define 
\begin{equation*}
\inf_{h\in H^1_f\setminus{\{0\}}} \frac{Q_f(h)}{\|h\|^2_{L^2_f}} =: \lambda_1(L_f) \quad \text{and } c_S:= \min \{1, \frac 12 \inf_M (R+m)\},
\end{equation*}
then $\lambda_1(L_f) \geq c_S >0.$ In particular, if condition \eqref{condition; R nonneg} holds and $m\geq 2,$ then $\lambda_1 \geq 1.$
\end{mtheorem}
\begin{rk}
If we only considered, as in Theorem \ref{thm; linear stability}, symmetric 2-tensors which are smooth and have compact support, the bounded curvature assumption is not actually a necessary part of $(\mathscr{E});$ see Section \ref{sec; stability} for more details. 
\end{rk}

As a direct corollary, the work of Deruelle \cite{DeruelleStability} gives the following dynamical stability result for the Ricci flow induced by an expanding gradient K\"ahler-Ricci soliton. We note that the assumptions will hold, for instance, for AC expanders.

\begin{mcor}\label{cor; dynamic stability}
Let $(M,g,\nabla f)$ be a complete, non-compact expanding gradient K\"ahler-Ricci soliton satisfying $|\Rm(g)| \leq K,$  and such that $f$ is proper. Then it is dynamically stable. More precisely, if $\bar g$ is a continuous metric on $M$ satisfying
\begin{equation*}
\|g-\bar g\|_{L^2_f} =I <\infty,
\end{equation*}
then there exists $\beta =\beta( n,I, K, \lambda_1)>0$ such that if $\bar g$ satisfies $\|g-\bar g\|_{L^\infty} \leq \beta,$ then there exists an immortal solution to the modified Ricci harmonic map heat flow starting from $\bar g$ and converging exponentially fast to $ g$ in the global $C^k$-sense, for any $k\in \mathbb N.$ Furthermore, the corresponding Ricci flow exists for all positive times and converges to $g$ after normalisation and pulling back by diffeomorphisms.
\end{mcor}
\begin{proof}
By \cite[Lemma 7]{Chan23}, bounded curvature and properness of $f$ imply Deruelle's assumption $(\mathscr{H})$ \cite{DeruelleStability}, and the result then follows from \cite[Theorems 1.2-1.3]{DeruelleStability}.  
\end{proof}

\subsection{Outline of the paper}
We briefly explain the main steps of the proofs of theorems \ref{thm; linear stability} and \ref{thm; strict stability}, while providing an outline of the paper. In Section \ref{sec; prelim} we fix some notation and the conventions for the rest of the paper, as well as provide some background results that will be useful later. In Section \ref{sec; action on sym tensors}, we study the space of symmetric 2-tensors on a K\"ahler-Ricci soliton. Given the K\"ahler condition, this space splits naturally into $J$-invariant and $J$-anti-invariant tensors, giving us two orthogonal subspaces. We can then study the action of $L_f$ on each of these subspaces. In Subsection \ref{subsection; J-invariant}, the action of the weighted Lichnerowicz operator is much simpler, and following the method in \cite{NaffOzuch} allows us to generalise their formulae in any dimension. 

In Subsection \ref{subsection; J-anti-invariant}, after identifying a $J$-anti-invariant 2-tensor with a suitable endomorphism of the tangent bundle $TM,$ we show that the unweighted version of the Lichnerowicz operator, $-\Delta -2\Rm\ast,$ sometimes also called the Einstein operator in the literature, can be identified with a complex Laplacian via a Weitzenb\"ock formula; this is essentially the work of Dai, Wang and Wei \cite{DWW} in the context of K\"ahler-Einstein metrics (see also \cite{KroenckePHD}), which was also used in \cite{HallMurphy} in the case of compact Ricci solitons, and goes back to the work of Koiso \cite{Koiso}. This will be the first main ingredient we use to prove our main results. We finish the section with Corollary \ref{cor; koiso on solitons}, providing a nice Weitzenb\"ock formula for $L_f$ on any gradient K\"ahler-Ricci soliton. In Section \ref{sec; stability}, we prove our main theorems. First, Lemma \ref{lem; conjugation} introduces the second main ingredient of the proof, a conjugation identity for $L_f.$ This identity is used by Deruelle in \cite{DeruelleStability} and allows us to study $L_f$ by studying a Schr\"odinger type operator, $-\Delta +V,$ where $V$ will be a particularly well-behaved potential depending on the curvature of $(M,g,\nabla f)$ and the potential function $f.$ 

Together with the Weitzenb\"ock formula from before, we then proceed to prove Theorem \ref{thm; linear stability} in the expanding and then in the steady case: the proofs are essentially the same, but we split them for clarity. Finally, we prove the strict stability of expanding K\"ahler-Ricci solitons under condition $\mathscr{E},$ Theorem \ref{thm; strict stability}, in Subsection \ref{subsection; strict stability}.

\subsection{Acknowledgements} The authors thank Alix Deruelle and Tristan Ozuch for helpful discussions. Both authors are funded by the German Research Foundation (DFG) – Project-ID 427320536 – SFB 1442, and by Germany’s Excellence Strategy EXC 2044/2 390685587, Mathematics Münster: Dynamics–Geometry–Structure.

\subsection{AI Disclosure} The writing and mathematical content of this paper is due to the authors. ChatGPT Pro was otherwise used only to assist with literature searches at the start of this project and to identify inconsistencies in notation and minor errors.

\section{Preliminaries}\label{sec; prelim}

\subsection{Notation and conventions}

We collect here the basic notation and conventions we use throughout the paper. For further details on the construction of some of the objects here, we refer the reader to \cite{Demailly} and \cite[Appendix B]{NaffOzuch}, as well as the references therein.
We adhere to the following curvature convention:
\begin{equation*}
\begin{split}
&\Rm(g)(U,V,W,Z)= g(([\nabla_U,\nabla_V]-\nabla_{[U,V]})W,Z),\\
&\Ric(g)(U,V)=\sum_i  \Rm(g)(U,e_i,e_i,V).
\end{split}
\end{equation*}
We write $\Rm \ast h$ for the curvature operator acting on symmetric 2-tensors; in coordinates, we define it by $(\Rm \ast h)_{ij} := R_{iklj} h^{kl}.$ It will be convenient to define the curvature operator $\mathcal{R}$ acting on a real 2-form $\alpha$ as 
\begin{equation}
(\mathcal{R}\alpha)(U,V) := -\sum_{i,j} \Rm(e_i,e_j, U,V)\alpha(e_i,e_j).
\end{equation}

The musical isomorphisms are defined as usual, with \(\flat : TM \to T^{\ast}M\) lowering indices, and \(\sharp : T^*M \to TM\) raising them. The commutator of two endomorphisms $A, B$ is given by $[A,B] = AB-BA.$ We write $\Delta = \tr_g\nabla^2=-\nabla^\ast\nabla$ and $h(U,V)=g(h^\sharp U,V)$. For a real $p$-form $\alpha$, the norm we use is 
\begin{equation*}
|\alpha|^2=(p!)^{-1}\sum_{i_1,\dots, i_p} |\alpha_{i_1\dots i_p}|^2.
\end{equation*}

The weighted divergence is defined by $\div_f(\cdot) = \div (\cdot) + \langle X, \cdot \rangle.$ For a symmetric 2-tensor $h$ we have
\begin{equation*}
(\div_f h)(U) = \sum_i (\nabla_{e_i} h)(e_i,U) + h(X,U).
\end{equation*}
The codifferential operator in coordinates is given by $\delta \alpha = - \sum_i \iota_{e_i} \nabla_{e_i} \alpha$ and its weighted version is defined by 
\begin{equation*}
\delta_f := e^{-f}\delta e^f = \delta - \iota_X.
\end{equation*}
It is then straightforward to see that $d$ and $\delta_f$ are formal adjoints in $L^2_f$. 

With the above, we define the Hodge Laplacian by $\Delta_H := -(d\delta +\delta d),$ and the weighted Hodge Laplacian by 
\begin{equation*}
\Delta_{H,f} := -(d\delta_f + \delta_f d) = \Delta_H + \lie_X.
\end{equation*}
For compactly supported smooth forms, $-\Delta_{H,f}$ is then a nonnegative operator.

On the holomorphic tangent bundle $T^{1,0}M,$ we define the Dolbeault Laplacian (see \cite[Chapter VI]{Demailly}) by
\begin{equation*}
\Box_{\bar \partial} := \bar\partial \bar{\partial}^\ast + \bar{\partial}^\ast\bar\partial.
\end{equation*}

\subsection{Steady and Expanding gradient Ricci solitons}

The following lemma states well-known identities for gradient Ricci solitons that will be useful to us later. Given that we are interested in K\"ahler-Ricci solitons in this work, we fix $n=2m$ so that $m$ is the complex dimension of $M.$ 
\begin{lemma}\label{id-SGE}
Let $(M^{2m},g,\nabla^g f)$ be a steady or expanding gradient Ricci soliton. Then:
\begin{eqnarray}
&&\Delta f = R+m\varepsilon, \label{equ:1} \\
&&\nabla R+ 2\Ric(g)(\nabla f)=0, \label{equ:2} \\
&&\arrowvert \nabla f \arrowvert^2+R=\varepsilon f+C_0, \label{equ:3}\\
&& \Delta_f R = -\varepsilon R -2|\Ric(g)|^2.
\end{eqnarray}
\end{lemma}
\begin{proof}
See \cite[Chapter 4]{ChowLuNi}.
\end{proof}
In the expanding case, after adding a constant to $f,$ we can always fix $C_0=0.$ For ease of notation, it will sometimes be convenient to denote the soliton vector field by $X$ instead of $\nabla f$ and we will use both interchangeably. 

\section{The space of symmetric 2-tensors on gradient K\"ahler-Ricci solitons}\label{sec; action on sym tensors}

For a gradient K\"ahler-Ricci soliton $(M,g,X)$ with complex structure $J,$ we denote the space of symmetric 2-tensors by $\Sym^2 := \Sym^2(T^{\ast}M).$ By setting $J \cdot h (U,V):= h(JU,JV),$ we can decompose $\Sym^2$ into the orthogonal subspaces corresponding to the $+1$ and $-1$ eigenvalues of $J\cdot$ (see \cite{KroenckePHD, NaffOzuch}). We define
\begin{equation}
\begin{split}
&\Sym_I^2 := \{ h\in \Sym^2 : J\cdot h=h\},\\
& \Sym_A^2 := \{ h\in \Sym^2 : J\cdot h=-h\},
\end{split}
\end{equation}
respectively called the $J$-invariant and the $J$-anti-invariant subspaces. For an arbitrary $h\in \Sym^2,$ we write $h=h_I + h_A$ for its decomposition. Since $J$ is parallel and $\Rm$ is $J$-invariant, it is straightforward to see that $\Delta, \Delta_f, \nabla,$ $\Rm\ast$ and $L_f$ all preserve this decomposition. 

\subsection{The $J$-invariant subspace}\label{subsection; J-invariant}

We start by analysing the action of $L_f$ on the $J$-invariant subspace. In this case, this is well-understood, so we will be brief. See also \cite{HallMurphy2, NaffOzuch}.

\begin{lemma}
The map \begin{equation}
\Sym^2_I\simeq \Lambda^{1,1}_\RR,\qquad h\mapsto h\circ J := h(J\cdot,\cdot),
\end{equation} is an isomorphism, where \(\Lambda^{1,1}_\RR\subset \Lambda^{1,1}\) is the bundle of real \((1,1)\)-forms on \(M\).
\begin{proof}
The map is clearly a linear isomorphism since it has the explicit inverse, \(\alpha\mapsto h:=-\alpha(J\cdot,\cdot)\); the non-trivial claim is that the image is precisely \(\Lambda^{1,1}_\RR\). The proof of this is the same as in the special case where \(h=g\) is the K\"ahler metric and \(\alpha=\omega\) is the K\"ahler form.
\end{proof}
\end{lemma}

Since \(\nabla J=0\), this map intertwines the rough Laplacian on the respective spaces: \(\Delta h \circ J = \Delta \alpha\). With our convention, we have $|h\circ J|^2 = \frac 12 |h|^2.$ Recall also (see e.g. \cite[Chapter 1, Section I]{Besse}) that the rough Laplacian on 2-forms is related to the Hodge Laplacian via the Weitzenb\"ock formula:\[\Delta_H \alpha =\Delta \alpha -(\Ric\cdot \alpha +\alpha\cdot \Ric)+\mathcal{R}(\alpha),\] where \((\Ric\cdot \alpha)_{ij} = \Ric_i^k\alpha_{kj}\) and \(( \alpha\cdot \Ric)_{ij} = \alpha_{ik}\Ric^k_j\) in local coordinates.

\begin{lemma}
If \(h\in \Sym^2_I\) and \(\alpha=h\circ J\), then \(\Rm*h\in \Sym^2_I\) and \[(2\Rm*h)\circ J= \mathcal{R}(\alpha). \]
\end{lemma}
\begin{proof}
The fact that $\Rm\ast h \in \Sym_I^2$ is straightforward from the K\"ahler symmetries (see \cite[Section 4.6]{KroenckePHD}). By the first Bianchi identity, we have
\begin{equation*}
\begin{split}
\mathcal{R}(\alpha)(U,V)&= \sum_{i,j} \Rm(Je_i,e_j,U,V)h_{ij} = -\sum_{i,j}(\Rm(e_j,U,Je_i,V) +\Rm(U,Je_i,e_j,V))h_{ij}\\
&=2\sum_{i,j}\Rm(JU,e_i,e_j,V)h_{ij}=2\Rm\ast h(JU,V),
\end{split}
\end{equation*}
where we also used that $h$ is $J$-invariant in the last line.
\end{proof}

\begin{prop}\label{prop; invariant formula}
Let \(h\in \Sym_{I}^2\) and \(\alpha=h\circ J \in \Lambda^{1,1}_{\RR}\). Then
\[-L_f h \circ J = \Delta_{H,f}\alpha -\varepsilon\alpha.\]
\end{prop}
\begin{proof}
\begin{align*}
-L_f h \circ J = (\Delta h +2\Rm*h +\nabla_{\nabla f} h)\circ J &= \Delta \alpha +\mathcal{R}(\alpha) +\nabla_{\nabla f}\alpha \\
&= \Delta_H \alpha + (\Ric\cdot \alpha+\alpha\cdot \Ric) +\nabla_{\nabla f} \alpha\\
&= \Delta_H \alpha -\varepsilon\alpha + (\nabla^2 f\cdot \alpha+\alpha\cdot \nabla^2 f) +\nabla_{\nabla f} \alpha\\
&=\Delta_H \alpha - \varepsilon\alpha + \mathcal{L}_{\nabla f}\alpha\\
&=\Delta_{H,f}\alpha -\varepsilon\alpha 
\end{align*}
where the second last equality follows from the identity \((\mathcal{L}_X-\nabla_X)\alpha=(\nabla^2 f)\cdot\alpha+\alpha\cdot\nabla^2 f\), which is a consequence of the torsion-free property of the Levi-Civita connection.

\end{proof} 

\subsection{The $J$-anti-invariant subspace}\label{subsection; J-anti-invariant}

For $h\in \Sym^2_A,$ note that $(h^\sharp)^\ast = h^\sharp$ and $h^\sharp J = -J h^\sharp.$ For each $J$-anti-invariant symmetric 2-tensor, we can then identify an endomorphism of $TM$ by 
\begin{equation*}
h \mapsto -Jh^\sharp =: H.
\end{equation*}
\begin{lemma}
For every $h\in \Sym_A^2,$ the associated endomorphism $H$ satisfies 
\begin{equation}\label{eq; properties of H}
H^\ast = H, \qquad HJ=-JH, \qquad h^\sharp =JH.
\end{equation}
Moreover, $|H| = |h|$ and $|\nabla H| = |\nabla h|.$
\end{lemma}

\begin{proof}
The proof is straightforward from the fact that $h$ is anti-invariant with respect to $J.$ First, we have 
\[
H^\ast =(-Jh^\sharp)^\ast = h^\sharp J=-Jh^\sharp = H
\]
 which shows the first equality in \eqref{eq; properties of H}. We also have $HJ= -Jh^\sharp J=-h^\sharp$ and $JH= J(-Jh^\sharp)= h^\sharp,$ which proves the other two equalities. Equality of the norms follows trivially from $J$ being parallel and an isometry.
\end{proof}

Complexifying $H$ to $T_{\mathbb{C}}M= TM\otimes \mathbb{C},$ since $HJ=-JH,$ we can view it as a $T^{1,0}M$-valued $(0,1)$-form. This identification allows us to consider $\Delta_C H,$ where $\Delta_C$ is the \textit{complex Laplacian} defined by 
\[
\Delta_C := 2\Box_{\bar \partial}.
\]
Note that the complex Laplacian is, therefore, a nonnegative operator \cite[p. 332]{Demailly}. The following proposition allows us to relate $\Delta_C$ to the Lichnerowicz Laplacian restricted to $\Sym^2_A.$ We refer to \cite[p. 55]{KroenckePHD}, \cite[Lemma 3.3]{HallMurphy} and \cite[Section 3]{DWW} for similar discussions.

\begin{prop}\label{prop; koiso}
Let $(M,g,J)$ be a K\"ahler manifold and let $h\in \Sym^2_A,$ with $H= -J h^\sharp.$ Then
\begin{equation}\label{eq; gen koiso}
\Delta_C H= J \left( (\Delta + 2\Rm\ast)h\right)^\sharp +  [ H, \Ric^\sharp ].
\end{equation}
In particular, 
\begin{equation}\label{eq; good koiso}
\langle \Delta_C H,H\rangle = \langle -(\Delta +2\Rm\ast)h,h\rangle.
\end{equation}
\end{prop}
\begin{proof}
In a normal unitary frame at a fixed point on $M,$ we have $(\bar \partial H)_{\bar i \bar j}^k = \nabla_{\bar i}H^k_{\bar j} - \nabla_{\bar j}H^k_{\bar i}$ and $(\bar{\partial}^{\ast}H)^k=-\nabla_i H_{\bar i}^k,$ therefore
\begin{equation*}
(\Box_{\bar \partial}H)_{\bar j}^k = -\nabla_i\nabla_{\bar i}H^k_{\bar j} + [\nabla_i,\nabla_{\bar j}]H_{\bar i}^k,
\end{equation*}
and
\begin{equation*}
-2\nabla_i\nabla_{\bar i}H = \nabla^{\ast}\nabla H -[\nabla_i,\nabla_{\bar i}]H.
\end{equation*}
For the commutator, we have $\left( [ \nabla_i, \nabla_{\bar j}] H\right)_{\bar l \bar k} = R_{i\bar j p \bar l}H_{\bar p \bar k} + R_{i\bar j p \bar k}H_{\bar l \bar p}.$ Using the K\"ahler symmetries and Ricci identities we get
\begin{equation}\label{eq; koiso1}
(\Delta_C H)_{\bar j\bar k}= (\nabla^{\ast}\nabla H)_{\bar j \bar k} +2R_{i \bar j p \bar k}H_{\bar i \bar p} + \Ric_{p\bar j}H_{\bar p \bar k} - \Ric_{p \bar k}H_{\bar j \bar p}.
\end{equation}

For the first term above, we note that, since $J$ is parallel, 
\begin{equation}\label{eq; koiso2}
\nabla^{\ast}\nabla H = -J(\nabla^{\ast}\nabla h)^\sharp = J(\Delta h)^\sharp.
\end{equation}
For the second term, we recall that, by definition, $(\Rm\ast h)_{\bar j \bar k}= -R_{i \bar j p \bar k}h_{\bar i \bar p}.$  Moreover, since $J= i$ on $T^{1,0}M$, we have $H_{\bar j \bar k}= -i h_{\bar j \bar k}.$ Thus,
\begin{equation}\label{eq; koiso3}
\begin{split}
g\left( J(2\Rm\ast h)^\sharp Z_{\bar j}, Z_{\bar k}\right) &= 2i (\Rm \ast h)_{\bar j \bar k} = -2i R_{i\bar j p \bar k}h_{\bar i \bar p} \\
& = 2R_{i \bar j p \bar k}H_{\bar i \bar p}.
\end{split}
\end{equation}
An analogous computation for $[H,\Ric^\sharp]$ deals with the Ricci term, and this together with \eqref{eq; koiso1}, \eqref{eq; koiso2} and \eqref{eq; koiso3} proves equation \eqref{eq; gen koiso}.

Finally, since $H$ and $\Ric^\sharp$ are both self-adjoint, their commutator is skew-adjoint and orthogonal to $H,$ which yields \eqref{eq; good koiso}.
\end{proof}

When $(M,g,J)$ is a gradient K\"ahler-Ricci soliton, equations \eqref{eq; gen koiso} and \eqref{eq; good koiso} further simplify as follows.
\begin{cor}\label{cor; koiso on solitons}
Let $(M,g,X)$ be a gradient K\"ahler-Ricci soliton. Then, with the notation introduced above, the following holds for every $h\in \Sym_A^2.$
\begin{equation}\label{eq; koiso on solitons}
-J(L_f h)^\sharp = \Delta_C H - \lie_X H.
\end{equation}
In particular, 
\begin{equation}\label{eq; good pairing}
\langle L_f h, h \rangle = \langle \Delta_C H, H\rangle - \frac 12 X( |H|^2).
\end{equation}
\end{cor}

\begin{proof}
Recalling that $L_f = -\Delta_f -2\Rm\ast = -\Delta -\nabla_{\nabla f} - 2\Rm\ast,$ Proposition \ref{prop; koiso} gives us
\begin{equation*}
\begin{split}
-J(L_f h)^\sharp &= J((\Delta + 2\Rm\ast)h)^\sharp + J(\nabla_{\nabla f}h)^\sharp \\
&= \Delta_C H -[H,\Ric^\sharp] - \nabla_X H.
\end{split}
\end{equation*}
The Lie derivative of the endomorphism $H$ is given by $\lie_X H = \nabla_X H + [H, \nabla X]$, while the soliton equation can be written as $\nabla X = \Ric^\sharp + \frac{\varepsilon}{2}\operatorname{Id}.$ Therefore, we have
\begin{equation*}
\begin{split}
-J(L_f h)^\sharp &= \Delta_C H -[H,\Ric^\sharp] - \nabla_X H \\
&= \Delta_C H - [H,\Ric^\sharp] - \lie_X H + [H,\Ric^\sharp] \\
&= \Delta_C H- \lie_X H.
\end{split}
\end{equation*}

For \eqref{eq; good pairing}, we recall that, since both $H$ and $\nabla X$ are self-adjoint, 
\begin{equation*}
\langle \lie_X H, H \rangle = \langle \nabla_X H, H \rangle + \langle [H, \nabla X ], H\rangle = \langle \nabla_X H, H \rangle.
\end{equation*}
Therefore, 
\begin{equation*}
\begin{split}
\langle L_f h, h \rangle &= \langle - J(L_f h)^\sharp, H \rangle = \langle \Delta_C H - \lie_X H, H\rangle \\
& = \langle  \Delta_C H, H\rangle - \langle \nabla_X H,H\rangle \\
&=  \langle \Delta_C H, H\rangle - \frac 12 X( |H|^2),
\end{split}
\end{equation*}
where, in the first equality, we used the fact that multiplying by $-J$ is an isometry.

\end{proof}

\section{Linear stability}\label{sec; stability}

\subsection{Conjugation of the Lichnerowicz operator and an unweighted estimate}

We recall the following conjugation identity for the operator $L_f$ \cite{DeruelleStability}, which shows that the operator $L_f$ is conjugate to a Schr\"odinger operator $-\Delta + V.$ This is crucially explored by Deruelle to study the spectrum of $L_f,$ as well as in the weighted estimates of \cite[Section 5]{BamlerChen}. Given it will be an important piece in the proof of our main theorem, we briefly sketch its proof. 

\begin{lemma}\label{lem; conjugation}
For any smooth $h\in \Sym^2,$ we have
\begin{equation}
\label{eq; conjugation}\left(e^{f/2} L_f e^{-f/2}\right) h = -(\Delta +2\Rm\ast)h + \frac 12\left( \Delta f + \frac 12 |\nabla f|^2 \right)h.
\end{equation}
\end{lemma}
\begin{proof}
Fixing $h \in C^{\infty}(M,\Sym^2 T^{\ast}M),$ a direct calculation gives
\begin{equation*}
\begin{split}
\left(e^{f/2} L_f e^{-f/2}\right)h & = e^{f/2} (-\Delta_f -2\Rm\ast)\left(e^{-f/2}h\right)\\
& = e^{f/2} \big [ -e^{-f/2}( \Delta h +2\Rm\ast h) - \nabla_{\nabla f} (e^{-f/2}h) - 2\langle \nabla(e^{-f/2}),\nabla h\rangle\\
& \qquad \quad -\Delta(e^{-f/2})h\big]\\
& = -(\Delta h + 2\Rm\ast h) - e^{f/2}\Delta_f(e^{-f/2})h-\nabla_{\nabla_f} h - 2e^{f/2}\langle \nabla(e^{-f/2}),\nabla h\rangle \\
& =  -(\Delta  + 2\Rm\ast )h + e^{f/2} h\bigg [ \frac 12 \left(\Delta f e^{-f/2} -\langle \nabla f, \frac 12 e^{-f/2} \nabla f\rangle \right)\\
&\quad  +\langle \nabla f, \frac 12 e^{-f/2} \nabla f\rangle \bigg] \\
&= -(\Delta  + 2\Rm\ast )h + \frac 12\left(\Delta f+ \frac 12 |\nabla f|^2\right)h.
\end{split}
\end{equation*}
\end{proof}

Therefore, Lemma \ref{lem; conjugation} relates the weighted operator $L_f$ with the unweighted Einstein operator $-(\Delta + 2\Rm\ast).$ Crucially, we have that this is a nonnegative operator when restricted to $J$-anti-invariant symmetric 2-tensors. In fact, recalling that for $h_A \in C_c^{\infty}(M, \Sym_A^2 T^{\ast}M),$ if $H_A = -J h_A^\sharp,$ then
\begin{equation*}
\langle \Delta_C H_A, H_A\rangle = \langle -(\Delta + 2\Rm\ast)h_A,h_A\rangle,
\end{equation*}
which yields, after integrating by parts, 
\begin{equation*}
\begin{split}
\int_M \left( |\nabla h_A|^2 -  \langle 2\Rm\ast h_A,h_A\rangle \right)d\mu_g &= \int_M \langle \Delta_C H_A,H_A\rangle d\mu_g\\
& = 4\left( \| \bar \partial H_A \|^2_{L^2} +  \| \bar \partial^\ast H_A \|^2_{L^2} \right) \geq 0,
\end{split}
\end{equation*}
where the norms in the last line are with respect to the standard Hermitian metric \cite[Chapter VI]{Demailly}. From the definition of stability in Definition \ref{def; linear stability}, the above gives $Q_0(h_A)\geq 0$ for every $h_A\in C_c^{\infty}(M, \Sym_A^2 T^{\ast}M)$.  

\subsection{The expanding case}

In this subsection, we combine the results from Section \ref{sec; action on sym tensors} and Lemma \ref{lem; conjugation} to prove Theorem \ref{thm; linear stability} in the expanding case. Let $h\in C_c^{\infty}(M, \Sym^2 T^{\ast}M).$ Then, we can write $h:= h_I + h_A,$ where $h_I\in \Sym_I^2$ and $h_A\in \Sym_A^2.$ From the discussion in Subsection \ref{subsection; J-invariant}  we have

\begin{equation*}
Q_f(h_I)=\int_M \langle L_f h_I, h_I\rangle d\mu_f = 2\int_M \langle L_f h_I\circ J, h_I\circ J \rangle d\mu_f.
\end{equation*}
Writing  $\alpha=h_I \circ J,$ Proposition \ref{prop; invariant formula} then yields
\begin{equation*}
\begin{split}
Q_f(h_I)
&=2\int_M \langle(-\Delta_{H,f}+1)\alpha,\alpha\rangle d\mu_f\\
&=2\int_M\langle -\Delta_{H,f}\alpha, \alpha \rangle d\mu_f +2\|\alpha\|_{L_f^2}^2\\
&=2\left(\|d\alpha\|^2_{L_f^2}+ \|\delta_f\alpha\|^2_{L_f^2}+ \|\alpha\|^2_{L_f^2}\right)
\geq \|h_I\|_{L_f^2}^2,
\end{split}
\end{equation*}
where we use the convention $|h_I|^2=2|\alpha|^2$ and integration by parts to deal with the weighted Hodge Laplacian term. This deals with $J$-invariant tensors.

In order to apply the conjugation identity and prove our main result, we define $k:= e^{f/2}h_A.$ Notice that we still have $k\in  C_c^{\infty}(M, \Sym_A^2 T^{\ast}M)$. Furthermore,

\begin{equation*}
\begin{split}
Q_f(h_A)
&=\int_M\langle L_fh_A,h_A\rangle e^f d\mu_g =\int_M\langle L_fe^{-f/2}k,e^{-f/2}k\rangle e^f d\mu_g\\
&=\int_M \langle (e^{f/2} L_f e^{-f/2}) k, k \rangle d\mu_g\\
&=\int_M \langle -(\Delta+2\Rm\ast)k, k \rangle d\mu_g
  +\frac12 \int_M \left( \Delta f+\frac12 |\nabla f|^2 \right) |k|^2 d\mu_g\\
&=Q_0(k)+\frac12 \int_M \left( \Delta f+ \frac12 |\nabla f|^2 \right) |k|^2 d\mu_g\\
&\geq \frac12 \int_M \left(R+m+ \frac12 |\nabla f|^2 \right) e^f |h_A|^2 d\mu_g\\
&\geq \frac12 \inf_M(R+m)\|h_A\|_{L_f^2}^2 \geq 0,
\end{split}
\end{equation*}
where we used that $R\geq -\frac n2=-m$ on an expanding gradient Ricci soliton \cite[Theorem 3]{PRS}. Since $L_f h = L_f h_I + L_f h_A,$ we finally get 
\begin{equation*}
\begin{split}
\int_M \langle L_f h,h\rangle d\mu_f &= \int_M \langle L_f h_I,h_I\rangle d\mu_f + \int_M \langle L_f h_A,h_A\rangle d\mu_f \\
& \geq \|h_I\|^2_{L_f^2} + \frac12\inf_M(R+m) \|h_A\|^2_{L_f^2} \\
&\geq \min\left\{1,\frac12\inf_M(R+m)\right\}\|h\|_{L^2_f}^2.
\end{split}
\end{equation*}

\subsection{The steady case}

Finishing the proof of Theorem \ref{thm; linear stability}, we now look at steady solitons. In this case, the identities in Lemma \ref{id-SGE} reduce to the following.
\begin{equation*}
\begin{split}
&\Ric = \nabla^2 f,\\
&\Delta f = R \geq 0,\\
&R + |\nabla f|^2 = C \geq0,
\end{split}
\end{equation*}
where $R\geq 0$ follows from \cite{ZhangSolitons}.

Writing $h = h_I + h_A$ as in the expanding case, we have $(L_f h_I)(J\cdot, \cdot) = -\Delta_{H,f} \alpha,$ so an analogous computation yields
\begin{equation*}
Q_f(h_I)= 2( \| d\alpha\|^2_{L_f^2} + \|\delta_f \alpha\|^2_{L_f^2}) \geq 0.
\end{equation*}
For $h_A \in C_c^{\infty}(M, \Sym_A^2 T^{\ast}M)$ we use the same conjugation formula, together with $k=e^{f/2}h_A$ and the soliton identities, to write
\begin{equation*}
\begin{split}
Q_f(h_A)& = \int_M \langle \Delta_C(e^{f/2}H_A), e^{f/2}H_A \rangle d\mu_g + \frac 12 \int_M \langle (\Delta f + \frac 12 |\nabla f|^2) k, k \rangle d\mu_g\\
&\geq \frac 12 \int_M (R+ \frac 12 |\nabla f|^2)|k|^2 d\mu_g = \frac 14 \int_M(R+C)|h_A|^2e^fd\mu_g \\
& \geq \frac C4 \|h_A\|^2_{L^2_f}.
\end{split}
\end{equation*}
Putting the two estimates together finally gives
\begin{equation*}
Q_f(h)= Q_f(h_I) + Q_f(h_A) \geq  \frac C4 \|h_A\|^2_{L^2_f} \geq 0.
\end{equation*}
\begin{rk}
Note that we do not immediately get strict stability from the argument above, as we can only show that $Q_f$ restricted to the $J$-invariant subspace is nonnegative. 
\end{rk}

\subsection{Strict stability of expanding K\"ahler-Ricci solitons}\label{subsection; strict stability}

We now look at the question of strict stability in the expanding case. For convenience, we state Theorem \ref{thm; strict stability} again.

\begin{thm}
Let $(M^{2m},g,\nabla f)$ be a non-compact, complete expanding gradient K\"ahler-Ricci soliton satisfying $(\mathscr{E})$. Then it is strictly linearly stable, with
\begin{equation}
\inf_{h\in H^1_f\setminus\{0\}} \frac{Q_f(h)}{\|h\|^2_{L^2_f}} =: \lambda_1(L_f)\geq  \min \{1, \frac 12 \inf_M (R+m)\}>0.
\end{equation}
In particular, if condition \eqref{condition; R nonneg} holds, i.e. $R\geq 0,$ and $m\geq 2,$ then $\lambda_1 \geq 1.$
\end{thm}

Before we prove the theorem, which will be a direct consequence of the calculations in the previous sections, we state the following technical proposition, which shows that the domain of the quadratic form 
\[
Q_f(h):= \int_M \left( |\nabla h|^2 - 2\langle \Rm\ast h,h\rangle\right)d\mu_f
\]
can be, in fact, extended to $H_f^1$ when the soliton is complete and has bounded curvature. Initially, we note that our Definition \ref{def; linear stability} only considers smooth symmetric 2-tensors with compact support, which is essential to justify all the integrations by parts throughout the paper, as well as to justify applying the conjugation identity \eqref{eq; conjugation} to the tensor $k=e^{f/2}h_A$ in the proof of Theorem \ref{thm; linear stability}. This is discussed in detail in \cite[Section 6.1]{DeruelleStability}, and for a standard reference on these types of results, see \cite[Chapter 6]{KatoBook}.

\begin{prop}\label{prop; form domain}
Let $(M,g,\nabla f)$ be a non-compact, complete steady or expanding gradient K\"ahler-Ricci soliton with bounded curvature. Then, the quadratic form 
\[
Q_f(h)= \int_M \left( |\nabla h|^2 - 2\langle \Rm\ast h,h\rangle\right)d\mu_f,\] defined for every $h \in C_c^{\infty}(M,\Sym^2 T^{\ast}M),$ extends to a nonnegative, closed quadratic form in $L^2_f,$ with domain equal to $H^1_f.$ In this setting, $Q_f$ is the closed quadratic form associated to the self-adjoint operator $L_f.$
\end{prop}
\begin{proof}
It is straightforward to see that when the soliton metric $g$ is complete, $C_c^{\infty}(M,\Sym^2 T^{\ast}M)$ is dense in $H^1_f$ with the norm 
\[
\| h\|^2_{H^1_f}:= \int_M (|h|^2 + |\nabla h|^2)d\mu_f,
\]
where $\nabla h$ is considered in the sense of distributions. Let $C_0>0$ be such that $|2\Rm\ast h|\leq C_0 |h|$ and consider a sequence $h_j \in C_c^{\infty}(M,\Sym^2 T^{\ast}M)$ with \( h_j \to h \) in \( H^1_f.\)
Then we have
\begin{equation*}
|Q_f(h_j)- Q_f(h)| \leq (1+C_0)\| h_j - h\|_{H^1_f}( \|h_j\|_{H^1_f} + \|h\|_{H^1_f}) \to 0
\end{equation*}
as $j\to \infty.$ In particular, 
\[Q_f(h)=\lim_{j\to \infty} Q_f(h_j) \geq 0.\]
A similar argument yields
\begin{equation*}
\| h\|_{H^1_f}^2\leq Q_f(h) + (1+C_0)\|h\|^2_{L^2_f} \leq (1+2C_0)\| h\|_{H^1_f}^2,
\end{equation*}
which shows that the norm $\left(Q_f(h) + (1+C_0)\|h\|^2_{L^2_f} \right)^{1/2}$ is equivalent to $\| h\|_{H^1_f}.$ In particular, $H_f^1$ is complete with this norm, which is enough to show closedness of $Q_f$ (see \cite[Chapter 6, Theorem 1.11]{KatoBook}). 
\end{proof}

\begin{proof}[Proof of Theorem \ref{thm; strict stability}]

First, we note that the proposition above justifies the infimum being taken over all non-zero symmetric 2-tensors in $H^1_f.$ Letting $u:= R+m,$ it follows from \cite[Theorem 3]{PRS} that $u\geq 0,$ and $u=0$ at one point implies that the metric is Einstein and $f$ is constant, which is excluded by each of the assumptions in $\mathscr{E}$. We can, therefore, assume $u> 0.$

Assuming \eqref{condition; f is proper}, we suppose, by way of contradiction, that $\inf_M u=0.$ The soliton identities from Lemma \ref{id-SGE} imply that $f+m = u + |\nabla f|^2 >0,$ and, since $f$ is proper by assumption, we have $\inf_M (f+m)>0.$ Again by the soliton identities, we get
\begin{equation*}
\Delta_f u = u - 2|\nabla^2 f|^2 \quad \text{ and } \Delta_f (f+m)= f+m.
\end{equation*}

For each small $\delta>0,$ we then consider $x_{\delta} \in M$ such that it minimises $u+ \delta(f+m),$ i.e.
\[
0 < u(x_{\delta}) + \delta(f(x_{\delta})+m) \leq u(x) + \delta(f(x) +m)
\]
for every $x\in M.$ Therefore, $u(x_{\delta}) + \delta(f(x_{\delta})+m)\to 0$ as $\delta \to 0$ and, at $x_{\delta},$ we have 
\begin{equation}\label{eq; hessian to 0}
2|\nabla^2 f|^2(x_{\delta}) \leq u(x_{\delta}) + \delta( f(x_{\delta}) +m) \to 0
\end{equation}
and $\nabla R(x_{\delta})= \nabla u(x_{\delta})=-\delta \nabla f(x_{\delta}).$ Since $\nabla R=-2\Ric(\nabla f) =  -2\nabla^2f(\nabla f, \cdot)+\nabla f,$ at $x_{\delta}$ we obtain
\begin{equation*}
\nabla^2 f(\nabla f,\cdot)= -\frac 12 \nabla R +\frac 12 \nabla f = \frac{1+\delta}{2}\nabla f.
\end{equation*}
If $|\nabla f|(x_{\delta})\neq 0,$ the equality above yields 
\begin{equation*}
|\nabla^2 f| \geq \frac{1+\delta}{2},
\end{equation*}
which contradicts \eqref{eq; hessian to 0}. Therefore, $|\nabla f|(x_{\delta})=0$ for every sufficiently small $\delta>0.$ This, however, implies that
\[
f(x_{\delta}) +m = u(x_{\delta}) \to 0,\] which contradicts the fact that $\inf (f+m)>0$ and proves our result.\\

Assumption \eqref{condition; R to 0} yields the conclusion after a straightforward application of the maximum principle to $u.$ Finally, \eqref{condition; R nonneg} also directly implies the theorem.

\end{proof}

\bibliographystyle{alphaurl}
\bibliography{stability}

\end{document}